\documentclass[11pt,letterpaper,reqno]{amsart}
\usepackage{tikz}
\usetikzlibrary{positioning, shapes.geometric, arrows.meta, calc}
\usepackage{amssymb}
\usepackage{amsmath}
\usepackage{amsthm}
\usepackage{amsfonts}
\usepackage{bbm}
\usepackage{enumitem} 
\usepackage{pgfplots}
\pgfplotsset{compat=1.18} 
\usepackage{booktabs}
\usepackage{graphicx}
\usepackage[T1]{fontenc}
\usepackage{doi}
\usepackage{float} 
\usepackage{bookmark}
\usepackage{hyperref}
\hypersetup{pdfstartview={FitH}}

\newtheorem{theorem}{Theorem}[section]
\newtheorem{lemma}[theorem]{Lemma}
\newtheorem{proposition}[theorem]{Proposition}
\newtheorem{corollary}[theorem]{Corollary}
\newtheorem{claim}{Claim}

\newtheorem{problem}[theorem]{Problem}

\theoremstyle{definition}

\newtheorem{remark}[theorem]{Remark}
\newtheorem{definition}[theorem]{Definition}
\numberwithin{equation}{section}

\newcommand{\cF}{\mathcal F}
\newcommand{\cE}{\mathcal E}

\newcommand{\cA}{\mathcal A}

\newcommand{\cH}{\mathcal H}
\newcommand{\cM}{\mathcal M}

\begin{document}

\title[Antichains with level multiplicity]{An Erd\H{o}s--Trotter problem on antichains with multiplicity $r$ on each occurring level}

\author[Y.~He]{Yixin He}
\author[Q.~Tang]{Quanyu Tang}

\address{School of Mathematical Sciences, Fudan University, Shanghai 200433, P. R. China}
\email{hyx717math@163.com}
\address{School of Mathematics and Statistics, Xi'an Jiaotong University, Xi'an 710049, P. R. China}
\email{tang\_quanyu@163.com}



\begin{abstract}
Fix an integer $r\ge2$. For each $n$ we consider families $\mathcal F\subseteq 2^{[n]}$ that form an antichain and have the property that, for every $t$, if there exists $A\in\mathcal F$ with $|A|=t$ then there exist at least $r$ members of $\mathcal F$ of size $t$. A problem of Erd\H{o}s and Trotter asserts that, for each fixed $r$, there exists a threshold $n_0(r)$ such that whenever $n>n_0(r)$ one can achieve $n-3$ distinct set sizes in such a family, and asks for estimates on $n_0(r)$.
We compute that $n_0(2)=3$ and $n_0(3)=8$.
For all $r\ge4$ we prove matching linear bounds up to lower-order terms, namely
$$
2r+2  \le n_0(r) \le 2r+2\log_2 r + O(\log_2\log_2 r).
$$
In particular, $n_0(r) = 2r + o(r)$.
\end{abstract}

\maketitle

\section{Introduction}

The study of antichains in the Boolean lattice $2^{[n]}$ (ordered by inclusion) is a cornerstone of extremal set theory, dating back to Sperner's theorem~\cite{Sp28}, which determines the maximum possible size of an antichain. Since then, antichains and their generalizations have been extensively investigated; see, e.g., the monographs of Anderson~\cite{An87} and Engel~\cite{Engel97}. While classical results such as the LYM inequality focus on bounding the cardinality of a family, Paul Erd\H{o}s often emphasized ``profile'' questions, asking which sets of cardinalities an antichain can realize under additional constraints.

In~\cite[Eq. (10.28)]{Er81} (see also~\cite[p.~119]{Gu83}), Erd\H{o}s and Trotter proposed the following problem, which also appears as Problem~\#776 on Bloom's Erd\H{o}s Problems website~\cite{EP776}.

\begin{problem}[\cite{Er81,Gu83}]\label{prob:EP776}
Let $S$ be a set with $|S| = n$. Consider a family $(A_k)$ of subsets of $S$ where no $A_k$ contains any other and such that, for every $t$, if there is an $A_k$ with $|A_k| = t$ then there are exactly $r$ different $A$'s of size $t$. If $r = 1$ and $n > 3$, one can always give $n - 2$ sets $A_k$ and one cannot give $n - 1$. If $r > 1$ and $n >$ some number $n_0 = n_0(r)$, one can always give $r(n-3)$ sets $A_k$ but one cannot give $r(n-2)$ such sets. 

Give estimates for $n_0(r)$.
\end{problem}

In~\cite{Gu83}, Guy writes:
\begin{quote}
\emph{[$\ldots$] We have no satisfactory estimate of $n_0(r)$. The content of the previous two paragraphs will be a small subset of a forthcoming paper of Erd\H{o}s, Szemer\'{e}di and Trotter.}
\end{quote}
To the best of our knowledge, we have not been able to locate the announced paper of
Erd\H{o}s, Szemer\'{e}di and Trotter, nor have we found any subsequent work in the literature
providing a satisfactory estimate for $n_0(r)$.

In this paper we obtain explicit upper and lower bounds for $n_0(r)$.

\subsection{A reformulation of Problem~\ref{prob:EP776}}\label{subsec:reformulation}
Fix integers $n\ge 1$ and $r\ge 1$.
For a family $\cF\subseteq 2^{[n]}$ write
\[
S(\cF):=\{\,|A|:\ A\in\cF\,\},
\qquad 
\cF_t:=\{A\in\cF:\ |A|=t\}\ \ (t=0,1,\dots,n).
\]
Call $\cF$ an \emph{$r$-multiplicity antichain} if it satisfies:
\begin{enumerate}[label=(\roman*)]
\item (\emph{Antichain}) $A\not\subseteq B$ for all distinct $A,B\in\cF$.
\item (\emph{$r$-multiplicity}) For every $t\in S(\cF)$ one has $|\cF_t|\ge r$.
\end{enumerate}
Following Guy's statement \cite{Gu83}, the underlying extremal quantity is
\[
g(n,r) := \max\bigl\{|S(\cF)|:\ \cF\subseteq 2^{[n]}\ \text{is an $r$-multiplicity antichain}\bigr\},
\]
i.e.\ the largest number of distinct cardinalities that can occur in an $r$-multiplicity antichain on $[n]$. Guy also reports that for $r>1$ there exists a threshold $n_0(r)$ such that, for all sufficiently large $n$,
one always has $g(n,r)=n-3$ (equivalently, $n-3$ distinct sizes are always achievable but $n-2$ are not).

\begin{remark}\label{rem:eq-vs-ge}
In the original formulation of Erd\H{o}s--Trotter~\cite[Eq.~(10.28)]{Er81}, one requires that for every occurring level $t\in S(\cF)$
there are \emph{exactly} $r$ sets of size $t$, i.e.\ $|\cF_t|=r$.
Guy~\cite{Gu83} relaxes this to $|\cF_t|\ge r$. For the extremal quantity $g(n,r)$ studied in this paper, these two conventions are equivalent.
Indeed, given any antichain $\cF$ with $|\cF_t|\ge r$ for all $t\in S(\cF)$,
select an arbitrary subfamily $\cF'_t\subseteq \cF_t$ of size exactly $r$ for each such $t$ and set
$\cF':=\bigcup_{t\in S(\cF)}\cF'_t$.
Then $\cF'$ is still an antichain, satisfies $|\cF'_t|=r$ for all $t\in S(\cF')$,
and preserves the set of occurring sizes: $S(\cF')=S(\cF)$.
Hence the maximum possible value of $|S(\cF)|$ is the same under either convention,
and we adopt Guy's inequality formulation for convenience.
\end{remark}

Accordingly, the threshold $n_0(r)$ appearing in Problem~\ref{prob:EP776} can be equivalently defined as follows.

\begin{definition}\label{def:threshold}
Fix an integer $r\ge2$. We define $n_0(r)$ to be the smallest integer $n_0$ such that
\[
g(n,r)=n-3
\]for every integer $n>n_0$.
\end{definition}

For small values of $r$ one can determine $n_0(r)$ exactly by computation.
In particular, $n_0(2)=3$ and $n_0(3)=8$; see Appendix~\ref{app:small-r} for details.

We now state our main results.

\begin{theorem}\label{thm:main_lower_bound}
For every integer $r\ge 4$, one has
\[
n_0(r) \ge 2r+2.
\]
\end{theorem}

\begin{theorem}\label{thm:main_upper_bound}
For every integer $r\ge 2$, one has
\[
n_0(r) \le 2r+2\log_2 r + O(\log_2\log_2 r).
\]
\end{theorem}

\subsection{Paper organization}
The remainder of the paper is organized as follows.
In Section~\ref{sec:main-lower-b} we prove Theorem~\ref{thm:main_lower_bound}, which provides a lower bound for $n_0(r)$.
In Section~\ref{sec:main-upper-b} we prove Theorem~\ref{thm:main_upper_bound}, which gives an upper bound for $n_0(r)$.
Section~\ref{sec:remarks} collects further remarks and open problems.
Finally, in Appendix~\ref{app:small-r} we determine the exact values $n_0(2)$ and $n_0(3)$.

\section{Preliminaries}

We collect several auxiliary statements that will be used in the proofs of our main theorems.

\begin{lemma}\label{lem:central-binomial-sqrt}
For every integer $m\ge 1$,
\[
\binom{m}{\lfloor m/2\rfloor} \ge \frac{2^m}{2\sqrt m}.
\]
\end{lemma}
\begin{proof}
This follows from the standard estimate for the central binomial coefficient derived from Stirling's formula. We omit the details.
\end{proof}

\begin{corollary}\label{cor:choose-m-K}
Let $K\ge4$ be an integer and let $m$ be the least integer such that
\[
\binom{m}{\lfloor m/2\rfloor}\ge K.
\]
Then
\begin{equation}\label{eq:choose-m-K}
m \le \Bigl\lceil \log_2 K + \frac12\log_2\log_2 K + 2\Bigr\rceil.
\end{equation}
\end{corollary}
\begin{proof}
Set \(m_0:=\Bigl\lceil \log_2 K + \frac12\log_2\log_2 K + 2\Bigr\rceil\). By Lemma~\ref{lem:central-binomial-sqrt},
\[
\binom{m_0}{\lfloor m_0/2\rfloor}\ \ge\ \frac{2^{m_0}}{2\sqrt{m_0}}.
\]
Moreover $2^{m_0}\ge 2^{\,\log_2 K+\frac12\log_2\log_2 K+2}=4K\sqrt{\log_2 K}$, so
\[
\frac{2^{m_0}}{2\sqrt{m_0}}
\ge 2K\sqrt{\frac{\log_2 K}{m_0}}.
\]
If $K\ge16$ then $m_0\le 2\log_2 K$, so the right-hand side is at least $\sqrt2\,K\ge K$.
For $4\le K\le 15$ the bound \eqref{eq:choose-m-K} is verified by a direct check.
Thus $\binom{m_0}{\lfloor m_0/2\rfloor}\ge K$, and by minimality of $m$ we have $m\le m_0$, proving \eqref{eq:choose-m-K}.
\end{proof}

We will use the following elementary classification of pairwise-intersecting families of $2$-subsets.

\begin{lemma}\label{lem:pairwise-intersecting-2sets}
Let $X$ be a finite set and let $\cE\subseteq \binom{X}{2}$ be a family of $2$-subsets
such that any two members of $\cE$ intersect. Then either
\begin{enumerate}[label=(\roman*)]
\item $\cE$ is a \emph{star}, i.e.\ there exists $a\in X$ such that $a\in E$ for every $E\in\cE$, or
\item $\cE$ is a \emph{triangle}, i.e.\ there exist distinct $a,b,c\in X$ such that
$\cE=\bigl\{\{a,b\},\{a,c\},\{b,c\}\bigr\}$.
\end{enumerate}
\end{lemma}

\begin{proof}
If $|\cE|\le 2$ then (i) holds trivially.
Assume $|\cE|\ge 3$ and pick two distinct edges $\{a,b\},\{a,c\}\in\cE$.
If every edge in $\cE$ contains $a$, then (i) holds and we are done.
Otherwise, there exists $E\in\cE$ with $a\notin E$.
Since $E$ intersects both $\{a,b\}$ and $\{a,c\}$, we must have $b,c\in E$,
so $E=\{b,c\}$. Now let $F\in\cE$ be arbitrary. Then $F$ intersects each of
$\{a,b\},\{a,c\},\{b,c\}$. If $a\in F$, then $F$ must contain $b$ or $c$,
hence $F\in\bigl\{\{a,b\},\{a,c\}\bigr\}$.
If $a\notin F$, then $F$ must intersect both $\{a,b\}$ and $\{a,c\}$, so again
$F=\{b,c\}$. Therefore $\cE=\bigl\{\{a,b\},\{a,c\},\{b,c\}\bigr\}$ and (ii) holds.
\end{proof}

We will also use the following simple ``label'' antichain gadget.

\begin{lemma}\label{lem:labels-general}
Let $m\ge 4$ and set $\ell:=\lfloor m/2\rfloor$.
Let $V_0$ be a set of size $m$ and let $b$ be an element not in $V_0$.
Put $V:=V_0\cup\{b\}$.
Let $K$ be an integer such that
\begin{equation}\label{eq:labels-general-K}
K\ \ge\ \ell+1
\qquad\text{and}\qquad
K-(\ell+1)\ \le\ \binom{m}{\ell}.
\end{equation}
Then there exist sets $L_t\subseteq V$ for all integers $t\in\{4,5,\dots,K\}$ such that
\begin{enumerate}[label=(\roman*)]
\item $|L_t|=2$ for every $t\in\{4,5,\dots,\ell+1\}$, and these $L_t$ are pairwise distinct and all contain $b$;
\item $|L_t|=\ell$ for every $t\in\{\ell+2,\ell+3,\dots,K\}$, and these $L_t$ are pairwise distinct and all lie in $V_0$;
\item the family $\{L_t:4\le t\le K\}$ is an antichain under inclusion.
\end{enumerate}
\end{lemma}

\begin{proof}
Since $\ell=\lfloor m/2\rfloor$, we have $\ell-2\le m$.
Choose distinct $v_1,\dots,v_{\ell-2}\in V_0$ (possible because $m\ge4$ implies $\ell\ge2$), and define
\[
L_t:=\{b,v_{t-3}\}\qquad (4\le t\le \ell+1).
\]
For $t\ge \ell+2$, we need $K-(\ell+1)$ many $\ell$-subsets of $V_0$.
By~\eqref{eq:labels-general-K}, we may choose distinct
\[
L_t\in \binom{V_0}{\ell}\qquad (\ell+2\le t\le K).
\]
Finally, inclusion cannot occur between distinct $L_t$:
within each of the two subfamilies all sets have the same size, so inclusion forces equality;
and no $L_t$ with $t\le \ell+1$ can be contained in an $L_s$ with $s\ge \ell+2$ because the former contains $b$
while the latter does not. This proves~(iii).
\end{proof}

Finally, we record a universal obstruction: for $r\ge2$ one cannot realize $n-2$ distinct levels. This fact is taken for granted in the statement of Problem~\ref{prob:EP776}; we include a complete proof here for completeness.

\begin{lemma}\label{lem:no-n-2}
Let $r\ge 2$ and $n\ge 4$. Then every $r$-multiplicity antichain $\cF\subseteq 2^{[n]}$ satisfies \(|S(\cF)|\ \le\ n-3\). Equivalently, for all $r\ge2$ and $n\ge4$, $$g(n,r)\le n-3.$$
\end{lemma}

\begin{proof}
Fix $r\ge2$ and $n\ge4$, and let $\cF\subseteq 2^{[n]}$ be an $r$-multiplicity antichain.

Since $r\ge2$, the levels $t=0$ and $t=n$ cannot occur: indeed $|\cF_0|\le1$ and $|\cF_n|\le1$,
contradicting the $r$-multiplicity condition whenever $0$ or $n$ lies in $S(\cF)$.
Hence
\[
S(\cF)\subseteq \{1,2,\dots,n-1\}.
\]

\begin{claim}\label{claim:2.4.1}
If $1\in S(\cF)$ then $n-1\notin S(\cF)$.
\end{claim}
\begin{proof}[Proof of Claim~\ref{claim:2.4.1}]
Indeed, since $1\in S(\cF)$ we have $|\cF_1|\ge r\ge2$, so choose two distinct singletons $\{x\},\{y\}\in\cF_1$.
Any $(n-1)$-subset of $[n]$ misses exactly one element, hence contains at least one of $x,y$.
Therefore every $(n-1)$-set $H$ satisfies $\{x\}\subseteq H$ or $\{y\}\subseteq H$, contradicting the antichain property.
So $\cF_{n-1}=\varnothing$, i.e.\ $n-1\notin S(\cF)$.
\end{proof}

\begin{claim}\label{claim:2.4.2}
If $n-1\in S(\cF)$ then $1\notin S(\cF)$.
\end{claim}
\begin{proof}[Proof of Claim~\ref{claim:2.4.2}]
Indeed, if $n-1\in S(\cF)$ then $|\cF_{n-1}|\ge r\ge2$, so pick two distinct $(n-1)$-sets
$H_1=[n]\setminus\{x_1\}$ and $H_2=[n]\setminus\{x_2\}$ with $x_1\ne x_2$.
For any singleton $\{z\}$, at least one of $H_1,H_2$ contains $z$ (since $z$ cannot equal both $x_1$ and $x_2$),
so $\{z\}\subseteq H_i$ for some $i\in\{1,2\}$, contradicting the antichain property.
Hence $\cF_1=\varnothing$, i.e.\ $1\notin S(\cF)$.
\end{proof}

Suppose for contradiction that $|S(\cF)|\ge n-2$.
Then by Claims~\ref{claim:2.4.1} and~\ref{claim:2.4.2} necessarily $|S(\cF)|=n-2$ and there exists a unique $t_0\in\{1,\dots,n-1\}$ such that
\[
S(\cF)=\{1,2,\dots,n-1\}\setminus\{t_0\}.
\]

Now we split into cases.

\smallskip
\noindent\underline{\emph{Case 1: $1\in S(\cF)$.}}
By Claim~\ref{claim:2.4.1} we have $n-1\notin S(\cF)$, so the unique missing level is $t_0=n-1$ and in particular $n-2\in S(\cF)$.

Choose $r$ distinct singletons $\{x_1\},\dots,\{x_r\}\in\cF_1$ and set $X=\{x_1,\dots,x_r\}$.
If $A\in\cF$ has $|A|\ge2$ then $A$ cannot contain any $x_i$, otherwise $\{x_i\}\subseteq A$ contradicts the antichain property.
Hence every $A\in\cF$ with $|A|\ge2$ satisfies $A\subseteq [n]\setminus X$, where $|[n]\setminus X|=n-r$.
In particular, every $(n-2)$-set in $\cF_{n-2}$ must be an $(n-2)$-subset of an $(n-r)$-set, so
\[
|\cF_{n-2}|\ \le\ \binom{n-r}{n-2}.
\]
If $r\ge3$ then $\binom{n-r}{n-2}=0$, and if $r=2$ then $\binom{n-r}{n-2}=\binom{n-2}{n-2}=1$.
In either case we get $|\cF_{n-2}|<r$, contradicting $n-2\in S(\cF)$ and $r$-multiplicity.

\smallskip
\noindent\underline{\emph{Case 2: $n-1\in S(\cF)$.}}
By Claim~\ref{claim:2.4.2} we have $1\notin S(\cF)$, so the unique missing level is $t_0=1$ and in particular $2\in S(\cF)$.

Choose $r$ distinct $(n-1)$-sets $H_i\in\cF_{n-1}$, say $H_i=[n]\setminus\{x_i\}$ with distinct $x_1,\dots,x_r$,
and set $X=\{x_1,\dots,x_r\}$.
Let $A\in\cF$ with $|A|\le n-2$. For each $i$, we must have $A\not\subseteq H_i$ by the antichain property.
But if $x_i\notin A$ then $A\subseteq [n]\setminus\{x_i\}=H_i$, a contradiction. Hence $x_i\in A$ for every $i$,
so $X\subseteq A$ and therefore $|A|\ge r$.

In particular, if $r\ge3$ then $\cF_2=\varnothing$, contradicting $2\in S(\cF)$.
If $r=2$ then every $2$-set in $\cF_2$ must equal $X=\{x_1,x_2\}$, so $|\cF_2|\le1<r$, again contradicting $r$-multiplicity.
This finishes the contradiction.

\smallskip
Therefore $|S(\cF)|\le n-3$ for every $r$-multiplicity antichain $\cF$, i.e.\ $g(n,r)\le n-3$.
\end{proof}

\section{Lower Bound}\label{sec:main-lower-b}

\begin{proposition}\label{prop:lower-2r+2}
Let $r\ge 4$ and let $n$ satisfy $r+3\le n\le 2r+2$.
Then every $r$-multiplicity antichain $\cF\subseteq 2^{[n]}$ satisfies
\[
|S(\cF)|\le n-4.
\]
\end{proposition}

\begin{proof}
Assume for contradiction that there exists an $r$-multiplicity antichain $\cF\subseteq 2^{[n]}$ with \(|S(\cF)|\ge n-3\). By Lemma~\ref{lem:no-n-2} every $r$-multiplicity antichain satisfies $|S(\cF)|\le n-3$.
Therefore $\cF$ must in fact satisfy
\[
|S(\cF)|=n-3.
\]Since there are $n+1$ possible sizes in $\{0,1,\dots,n\}$, the equality $|S(\cF)|=n-3$ means that exactly four sizes are missing.

Since $r\ge 4$, the levels $0$ and $n$ cannot occur in $S(\cF)$ (each contains only one set).
If $1\in S(\cF)$, then $|\cF_1|\ge r$ and hence $\cF_1$ contains $r$ distinct singletons
$\{x_1\},\dots,\{x_r\}$. By the antichain property, no set in $\cF$ may contain any $x_i$; thus every
$A\in\cF$ is contained in $[n]\setminus\{x_1,\dots,x_r\}$ and so $|A|\le n-r$.
In particular, none of the sizes $n-r+1,\dots,n$ occurs. Together with $0\notin S(\cF)$, this forces at least
$r+1\ge 5$ missing sizes, contradicting that exactly four sizes are missing.
Hence $1\notin S(\cF)$.

Now consider the complement family $\cF^c:=\{[n]\setminus A:\ A\in\cF\}$.
It is again an $r$-multiplicity antichain, and $S(\cF^c)=\{n-t:\ t\in S(\cF)\}$. Since $|S(\cF^c)|=|S(\cF)|=n-3$, the same argument shows that $1\notin S(\cF^c)$. Thus $n-1\notin S(\cF)$. Therefore the missing sizes are precisely $\{0,1,n-1,n\}$ and hence
\begin{equation}\label{eq:levels-2-to-n-2-2r+2}
S(\cF)=\{2,3,\dots,n-2\}.
\end{equation}

Let
\[
\cM:=\cF_{n-2}^c=\{\, [n]\setminus A:\ A\in\cF_{n-2}\,\}\subseteq \binom{[n]}{2}.
\]
Then $|\cF_2|\ge r$ and $|\cM|\ge r$ by $r$-multiplicity and \eqref{eq:levels-2-to-n-2-2r+2}.
Moreover, for all $E\in\cF_2$ and $M\in\cM$ we have
\begin{equation}\label{eq:cross-int-2r+2}
E\cap M\neq\varnothing,
\end{equation}
since otherwise $E\subseteq [n]\setminus M\in\cF_{n-2}$ would contradict that $\cF$ is an antichain.

We claim that $\cF_2$ is pairwise intersecting.
Indeed, if $\cF_2$ contained two disjoint edges $\{a,b\}$ and $\{c,d\}$, then any $M\in\cM$ must intersect both,
so necessarily
\[
\cM\subseteq \bigl\{\{a,c\},\{a,d\},\{b,c\},\{b,d\}\bigr\},
\]
whence $|\cM|\le 4$.
If $r>4$ this contradicts $|\cM|\ge r$.
If $r=4$, then $|\cM|\ge 4$ forces equality, so
$\cM=\{\{a,c\},\{a,d\},\{b,c\},\{b,d\}\}$.
But then the only $2$-sets intersecting all members of $\cM$ are $\{a,b\}$ and $\{c,d\}$, so $|\cF_2|\le 2$,
contradicting $|\cF_2|\ge r=4$.
This proves that $\cF_2$ is pairwise intersecting.

We claim that $\cF_2$ is a star, and all $(n-2)$-sets avoid its centre. Indeed, by Lemma~\ref{lem:pairwise-intersecting-2sets}, a pairwise intersecting family of $2$-sets is either a star or a triangle. Since $|\cF_2|\ge r\ge 4$,
it cannot be a triangle. Hence there exists a point $x\in[n]$ and a set $P\subseteq [n]\setminus\{x\}$ with
$|P|=|\cF_2|\ge r$ such that
\begin{equation}\label{eq:star-2r+2}
\cF_2=\{\{x,p\}:\ p\in P\}.
\end{equation}
Next we show that every $M\in\cM$ contains $x$.
Fix $M\in\cM$. By \eqref{eq:cross-int-2r+2} and \eqref{eq:star-2r+2}, $M$ must intersect $\{x,p\}$ for every $p\in P$.
If $x\notin M$, then $M$ is a $2$-set and must contain every $p\in P$, impossible since $|P|\ge r\ge 4>2$.
Therefore $x\in M$ for all $M\in\cM$. Writing $M=\{x,q\}$, we obtain a set
\[
Q:=\{\,q\in[n]\setminus\{x\}:\ \{x,q\}\in\cM\,\}
\quad\text{with}\quad |Q|=|\cM|\ge r,
\]
and equivalently every $A\in\cF_{n-2}$ has the form $A=[n]\setminus\{x,q\}$ for some $q\in Q$.
In particular,
\[
x\notin A \qquad\text{for all }A\in\cF_{n-2}.
\]

Let $2\le t\le r-1$ and take any $B\in\cF_t$ (which exists by \eqref{eq:levels-2-to-n-2-2r+2}).
Suppose $x\notin B$. Since $|B|<|Q|$, we know that $Q\setminus B\neq\varnothing$. For any $q\in Q\setminus B$ we would have
\[
B\subseteq [n]\setminus\{x,q\}\in\cF_{n-2},
\]
contradicting that $\cF$ is an antichain. This proves
\begin{equation}\label{eq:small-levels-contain-x-fixed}
x\in B\quad\text{for every }B\in\cF_t\ \text{and every }2\le t\le r-1.
\end{equation}

Set
\[
U:=[n]\setminus\bigl(\{x\}\cup P\bigr).
\]
Since $|P|\ge r$, we have
\begin{equation}\label{eq:U-upper}
|U|=n-1-|P|\le n-1-r\le (2r+2)-1-r=r+1.
\end{equation}
Moreover, any set $B\in\cF$ with $x\in B$ and $|B|\ge 3$ cannot contain any $p\in P$,
since otherwise $\{x,p\}\in\cF_2$ would be contained in $B$, contradicting the antichain property.
Thus every $B\in\cF_t$ with $x\in B$ and $t\ge 3$ can be written as
\[
B=\{x\}\cup B',\qquad B'\subseteq U,\qquad |B'|=t-1.
\]

\smallskip
\noindent\underline{\emph{Case 1: $|U|\le r$.}}
By \eqref{eq:levels-2-to-n-2-2r+2} we have $r\in S(\cF)$, so $|\cF_r|\ge r$. 

We claim that there is at most one set $A\in\cF_r$ with $x\notin A$. Indeed, let $A\in\cF_r$ and suppose $x\notin A$.
If there existed $q\in Q\setminus A$, then $A\subseteq [n]\setminus\{x,q\}\in\cF_{n-2}$, contradicting that $\cF$ is an antichain. Hence $Q\subseteq A$, so $r=|A|\ge |Q|$.
On the other hand, $|Q|=|\cM|=|\cF_{n-2}|\ge r$, and thus $|Q|=r$ and necessarily $A=Q$. This proves the claim. 

Thus, we can choose
$r-1$ distinct sets $A_1,\dots,A_{r-1}\in\cF_r$ with $x\in A_i$.
Write $A_i=\{x\}\cup X_i$ where $X_i\subseteq U$ and $|X_i|=r-1$.
Since the $A_i$ are distinct, the sets $X_1,\dots,X_{r-1}$ are distinct.
This forces $|U|\ge r$, and together with $|U|\le r$ we get $|U|=r$. 

Now $|\cF_{r-1}|\ge r$ and by \eqref{eq:small-levels-contain-x-fixed} every set in $\cF_{r-1}$ contains $x$,
so pick $r$ distinct sets $B_1,\dots,B_r\in\cF_{r-1}$ and write
$B_j=\{x\}\cup Y_j$ where $Y_j\subseteq U$ and $|Y_j|=r-2$.
Since $|U|=r$, each $Y_j$ has exactly two supersets of size $r-1$ in $U$.
But $U$ has exactly $\binom{r}{r-1}=r$ many $(r-1)$-subsets, and the $r-1$ distinct sets
$X_1,\dots,X_{r-1}$ miss at most one of them; therefore, for each $j$ at least one of the two
$(r-1)$-supersets of $Y_j$ must equal $X_i$ for some $i$.
Hence for each $j$ there exist $i$ with $Y_j\subseteq X_i$, which implies
\[
B_j=\{x\}\cup Y_j\subseteq \{x\}\cup X_i=A_i,
\]
contradicting that $\cF$ is an antichain.

\smallskip
\noindent\underline{\emph{Case 2: $|U|=r+1$.}}
Then by~\eqref{eq:U-upper} we know that $n=2r+2$. Since $r\ge 4$, we have $3\le r-1$ and thus $\cF_3\neq\varnothing$ and every $3$-set in $\cF$ contains $x$
by \eqref{eq:small-levels-contain-x-fixed}. In particular, every $C\in\cF_3$ has the form
$C=\{x\}\cup e$ with $e\in\binom{U}{2}$.
Define
\[
\cH:=\bigl\{\,e\in\binom{U}{2}:\ \{x\}\cup e\in\cF_3\,\bigr\}.
\]
Then $|\cH|=|\cF_3|\ge r$.

\begin{claim}\label{claim:3.1}
At most one set in $\cF_{r+1}$ contains $x$.   
\end{claim}
\begin{proof}[Proof of Claim~\ref{claim:3.1}]
Indeed, suppose $A\in\cF_{r+1}$ with $x\in A$.
Then $A=\{x\}\cup Z$ with $Z\subseteq U$ and $|Z|=r$.
Since $|U|=r+1$, we can write $Z=U\setminus\{u\}$ for a unique $u\in U$.
If there exists $e\in\cH$ with $e\subseteq Z$, then $\{x\}\cup e\subseteq A$, contradicting the antichain property.
Thus for every $e\in\cH$ we have $e\nsubseteq Z$, which is equivalent to $u\in e$.
Hence $u\in \bigcap_{e\in\cH} e$.
But $\cH$ contains at least two distinct $2$-sets (since $|\cH|\ge r\ge 4$),
so $\bigcap_{e\in\cH} e$ has size at most $1$.
Therefore there is at most one possible $u$, and hence at most one such $A$.
\end{proof}

By Claim~\ref{claim:3.1} and $|\cF_{r+1}|\ge r$, there are at least $r-1$ distinct sets
$\widetilde{A}_1,\dots,\widetilde{A}_{r-1}\in\cF_{r+1}$ with $x\notin \widetilde{A}_i$.
Fix such an $\widetilde{A}_i$. If there exists $q\in Q$ with $q\notin \widetilde{A}_i$, then
$\widetilde{A}_i\subseteq [n]\setminus\{x,q\}\in\cF_{n-2}$ and the inclusion is proper because
$|\widetilde{A}_i|=r+1<(2r+2)-2=n-2$, contradicting that $\cF$ is an antichain.
Hence $Q\subseteq \widetilde{A}_i$ for all $i$, and in particular $|Q|\le r+1$.

Since we have at least $r-1$ distinct sets $\widetilde{A}_i$ of size $r+1$ all containing $Q$, it follows that $|Q|\neq r+1$.
Therefore
\begin{equation}\label{eq:Q-size-r}
|Q|=r.
\end{equation}
Writing $R:=[n]\setminus(\{x\}\cup Q)$, we have $|R|=n-1-|Q|=(2r+2)-1-r=r+1$ and each $\widetilde{A}_i$ is of the form
\[
\widetilde{A}_i=Q\cup\{a_i\}\qquad (a_i\in R),
\]
with $a_1,\dots,a_{r-1}$ pairwise distinct.

\begin{claim}\label{claim:3.2}
Every set in $\cF_{r+2}$ avoids $x$ and contains $Q$.
\end{claim}
\begin{proof}[Proof of Claim~\ref{claim:3.2}]
First, suppose $B\in\cF_{r+2}$ contains $x$.
Then $|B|\ge 3$ and hence $B=\{x\}\cup B'$ with $B'\subseteq U$ and $|B'|=r+1$.
Since $|U|=r+1$, this forces $B' = U$ and hence $B=\{x\}\cup U$.
But then every $C\in\cF_3$ satisfies $C\subseteq \{x\}\cup U=B$, contradicting the antichain property.
Thus $x\notin B$ for all $B\in\cF_{r+2}$. 

Now take any $B\in\cF_{r+2}$.
If there exists $q\in Q$ with $q\notin B$, then $B\subseteq [n]\setminus\{x,q\}\in\cF_{n-2}$ and the inclusion is proper
because $|B|=r+2<(2r+2)-2=n-2$, again contradicting that $\cF$ is an antichain.
Hence $Q\subseteq B$ for all $B\in\cF_{r+2}$, proving Claim~\ref{claim:3.2}.
\end{proof}

By \eqref{eq:Q-size-r} and Claim~\ref{claim:3.2}, every $B\in\cF_{r+2}$ has the form
\[
B=Q\cup\{b,b'\}\qquad (b,b'\in R,\ b\neq b').
\]
Since $|\cF_{r+2}|\ge r$, we can choose $r$ distinct such pairs $\{b,b'\}$.

However, for each $i\in\{1,\dots,r-1\}$ and each $B\in\cF_{r+2}$ we must have $\widetilde{A}_i\not\subseteq B$.
Equivalently, none of the pairs $\{b,b'\}$ may contain any $a_i$.
Thus every pair corresponding to a member of $\cF_{r+2}$ must lie in $R\setminus\{a_1,\dots,a_{r-1}\}$,
which has size at most $2$.
Hence there is at most one possible pair, contradicting $|\cF_{r+2}|\ge r\ge 4$.

This contradiction completes Case~2, and hence no $r$-multiplicity antichain on $[n]$ can satisfy $|S(\cF)|=n-3$
when $r+3\le n\le 2r+2$.
Therefore $|S(\cF)|\le n-4$ for all such $n$. 
\end{proof}

We are now ready to present the following.

\begin{proof}[Proof of Theorem~\ref{thm:main_lower_bound}]
By Proposition~\ref{prop:lower-2r+2}, in particular at $n=2r+2$ we have $g(2r+2,r)\le (2r+2)-4<(2r+2)-3$. By Definition~\ref{def:threshold}, this forces $n_0(r)\ge 2r+2$.    
\end{proof}

\section{Upper Bound}\label{sec:main-upper-b}

\begin{proposition}\label{thm:n0-upper-2r-logr}
For every integer $r\ge 2$ and every integer $n$ satisfying
\begin{equation}\label{eq:thm:n0-upper-2r-logr-disp_lower_b}
n\ge 2r+2\log_2 r+\log_2\log_2 r+15,
\end{equation}
one has $g(n,r)=n-3$. 
\end{proposition}

\begin{proof}
Fix $r\ge2$ and an integer $n$ satisfying the displayed lower bound~\eqref{eq:thm:n0-upper-2r-logr-disp_lower_b}, and set \(k:=\lfloor n/2\rfloor\). Let $m\ge4$ be the least integer such that
\begin{equation}\label{eq:choose-m-k-final}
\binom{m}{\lfloor m/2\rfloor}\ \ge\ k,
\end{equation}
and put $\ell:=\lfloor m/2\rfloor$. By Corollary~\ref{cor:choose-m-K} (applied with $K=k$) we have
\begin{equation}\label{eq:m-asymp-final}
m \le \log_2 k+\frac12\log_2\log_2 k+3.
\end{equation}

We next verify that
\begin{equation}\label{eq:room-k-r-m-final}
k\ge r+m+1.
\end{equation}

\medskip
\noindent\underline{\emph{Case 1: $n<8r$.}}
Then $k=\lfloor n/2\rfloor\le 4r$.
Under the hypothesis~\eqref{eq:thm:n0-upper-2r-logr-disp_lower_b} this case can only occur when $r\ge4$ (since for $r=2,3$ the right-hand side of~\eqref{eq:thm:n0-upper-2r-logr-disp_lower_b} is at least $8r$).
Using~\eqref{eq:m-asymp-final} and $k\le4r$ we obtain
\[
m+1\le \log_2(4r)+\frac12\log_2\log_2(4r)+4
=\log_2 r+6+\frac12\log_2(\log_2 r+2).
\]
On the other hand,
\[
k-r\ \ge\ \frac{n-2r-1}{2}
\ \ge\ \log_2 r+\frac12\log_2\log_2 r+\frac{15-1}{2}.
\]
Since $r\ge4$ implies $\log_2\log_2 r+2\ge \log_2(\log_2 r+2)$, the last display gives $k-r\ge m+1$, which is equivalent to \eqref{eq:room-k-r-m-final}.

\medskip
\noindent\underline{\emph{Case 2: $n\ge8r$.}}
Then $k\ge4r\ge8$, so $k-r\ge \frac34k$.
By minimality of $m$ we have $k>\binom{m-1}{\lfloor(m-1)/2\rfloor}\ge\binom{m-1}{2}$, hence $k\ge (m-1)(m-2)/2+1$.
If $m\leq 5$ then $m+1\leq 6\le\frac34k$ since $k\ge8$.
If $m\ge6$ then one checks that $\frac38(m-1)(m-2)\ge m+1$, so
\[
k-r\ \ge\ \frac34k\ \ge\ \frac38(m-1)(m-2)\ \ge\ m+1,
\]
and again \eqref{eq:room-k-r-m-final} follows.

Choose disjoint sets
\[
[n]=\{a\}\ \dot\cup\ P\ \dot\cup\ \{u\}\ \dot\cup\ V\ \dot\cup\ R\ \dot\cup\ W
\]
with $|P|=r$, $|V|=m+1$, $|R|=k-2$, and $W$ the (possibly empty) remaining set
(which exists by \eqref{eq:room-k-r-m-final}).  Write
\[
P=\{p_1,\dots,p_r\},\qquad V=V_0\cup\{b\},\qquad |V_0|=m,
\]
and set $R':=R\cup W$, so that $|R'|\ge |R|=k-2\ge r$ (using \eqref{eq:room-k-r-m-final} and $m\ge4$).

Apply Lemma~\ref{lem:labels-general} with $K=k$.
The hypothesis \eqref{eq:choose-m-k-final} implies $\binom{m}{\ell}\ge k$ and hence
$k-(\ell+1)\le \binom{m}{\ell}$, while $k\ge \ell+1$ holds since $m\ge4$ and $k\ge r+m+1$ by
\eqref{eq:room-k-r-m-final}.
Therefore we obtain sets $L_t\subseteq V$ for every $t\in\{4,5,\dots,k\}$ such that:
\begin{enumerate}[label=(\roman*)]
\item $|L_t|=2$ for $4\le t\le \ell+1$, and these $L_t$ are distinct and all contain $b$;
\item $|L_t|=\ell$ for $\ell+2\le t\le k$, and these $L_t$ are distinct and all lie in $V_0$;
\item $\{L_t:4\le t\le k\}$ is an antichain under inclusion.
\end{enumerate}
Define $\widetilde{\cF}$ level-by-level as follows.
\begin{itemize}
\item \emph{Size $2$:} for $j\in[r]$, set $F_{2,j}:=\{a,p_j\}$.
\item \emph{Size $3$:} choose distinct $x_1,\dots,x_r\in R'$ and set $F_{3,j}:=\{a,u,x_j\}$.
\item \emph{Size $t\in\{4,\dots,k\}$:} let $s(t):=t-1-|L_t|$.
Note that $s(t)\ge1$ and, since $\ell\ge2$ and $|R'|\ge k-2$,
\[
s(t)\ \le\ t-1-2\ \le\ k-3\ \le\ |R'|-1.
\]
Hence $1\le s(t)\le |R'|-1$ and therefore
\[
\binom{|R'|}{s(t)}\ \ge\ |R'|\ \ge\ r,
\]
so we can choose $r$ distinct $s(t)$-subsets $B_{t,1},\dots,B_{t,r}\subseteq R'$.
Define
\[
F_{t,j}:=\{a\}\cup L_t\cup B_{t,j}\qquad (j\in[r]).
\]
\end{itemize}Let
\[
\widetilde{\cF}:=\{F_{t,j}: 2\le t\le k,\ j\in[r]\}.
\]
Then every $F\in\widetilde{\cF}$ contains $a$, and for each $t\in\{2,3,\dots,k\}$ the family $\widetilde{\cF}$
contains at least $r$ sets of size $t$.

We next check that $\widetilde{\cF}$ is an antichain.
If $F_{t,j}\subseteq F_{t',j'}$ with $t,t'\ge4$, then by disjointness of the parts
\[
L_t\subseteq L_{t'}\quad\text{and}\quad B_{t,j}\subseteq B_{t',j'}.
\]
If $t\neq t'$, this contradicts the label antichain property (iii). If $t=t'$, then
$|B_{t,j}|=|B_{t,j'}|=s(t)$ and $B_{t,j}\subseteq B_{t,j'}$ forces $B_{t,j}=B_{t,j'}$,
contradicting distinctness. Finally, the cases involving sizes $2$ or $3$ cannot create inclusion
because $P$, $\{u\}$, and $R'$ are disjoint and each of $F_{2,j},F_{3,j}$ contains an element
($p_j$ or $u$) that never appears in any $F_{t',j'}$ with $t'\ge4$.
Thus $\widetilde{\cF}$ is an antichain. Moreover, $|F|\le k\le n/2$ for all $F\in\widetilde{\cF}$.

Define
\[
\cA:=\widetilde{\cF}\cup\{[n]\setminus F:\ F\in\widetilde{\cF}\}.
\]
Since $\widetilde{\cF}$ is an antichain, so is its complement family $\{[n]\setminus F:F\in\widetilde{\cF}\}$.
It remains to rule out inclusions between $\widetilde{\cF}$ and its complement family.
Fix $F,G\in\widetilde{\cF}$. Because every set in $\widetilde{\cF}$ contains $a$, the complement $[n]\setminus G$
does not contain $a$, hence $F\not\subseteq [n]\setminus G$.
For the reverse inclusion, note that
\[
|[n]\setminus G|\ =\ n-|G|\ \ge\ n-k\ =\ \lceil n/2\rceil\ \ge\ k,
\qquad\text{whereas}\qquad
|F|\ \le\ k.
\]
If $n$ is odd then $|[n]\setminus G|\ge k+1>|F|$, so $[n]\setminus G\not\subseteq F$.
If $n$ is even, then either $|[n]\setminus G|>k\ge |F|$, in which case $[n]\setminus G\not\subseteq F$ by cardinality,
or else $|[n]\setminus G|=k$ (equivalently $|G|=k$). In the latter case,
$[n]\setminus G\subseteq F$ would force $[n]\setminus G=F$, contradicting $a\in F$ and $a\notin [n]\setminus G$. Therefore $\cA$ is an antichain.

By construction, for each $t\in\{2,\dots,k\}$ there are at least $r$ sets of size $t$ in $\cA$,
and taking complements yields at least $r$ sets of each size
$t\in\{n-k,\dots,n-2\}$.
Since $k=\lfloor n/2\rfloor$, we have $n-k=\lceil n/2\rceil\in\{k,k+1\}$ and hence
\[
\{2,\dots,k\}\ \cup\ \{n-k,\dots,n-2\}\ =\ \{2,3,\dots,n-2\}.
\]
Therefore every size $t\in\{2,3,\dots,n-2\}$ appears in $\cA$ at least $r$ times, so
$|S(\cA)|=n-3$ and $\cA$ is an $r$-multiplicity antichain.

Thus, we construct an $r$-multiplicity antichain $\cA\subseteq 2^{[n]}$ with $|S(\cA)|=n-3$, hence \(g(n,r) \ge n-3\). On the other hand, since $r\ge2$ and the present hypothesis on $n$ implies $n\ge4$,
Lemma~\ref{lem:no-n-2} yields the universal upper bound $g(n,r)\le n-3$. Therefore $g(n,r)=n-3$.
\end{proof}

We are now ready to present the following.

\begin{proof}[Proof of Theorem~\ref{thm:main_upper_bound}]
By Proposition~\ref{thm:n0-upper-2r-logr}, we know that $g(n,r)=n-3$ for every integer $r\ge 2$ and every integer $n\ge 2r+2\log_2 r+\log_2\log_2 r+15$. By Definition~\ref{def:threshold}, this implies \begin{equation}\label{eq:logloglog_upper1}
n_0(r)\le 2r+2\log_2 r+\log_2\log_2 r+15.
\end{equation}In particular, one has \(n_0(r) \le 2r+2\log_2 r + O(\log_2\log_2 r)\).
\end{proof}

\section{Concluding remarks}\label{sec:remarks}

In this paper we studied the threshold $n_0(r)$ in the Erd\H{o}s--Trotter problem (Problem~\ref{prob:EP776}). For $r=2$ and $r=3$ one has $n_0(2)=3$ and $n_0(3)=8$.
For all integers $r\ge 4$, Theorem~\ref{thm:main_lower_bound} and
Theorem~\ref{thm:main_upper_bound} yield the two-sided estimate
\[
2r+2 \le n_0(r) \le 2r+2\log_2 r + O(\log_2\log_2 r).
\]
In particular, this determines the leading linear term of $n_0(r)$ as $2r$.

A natural next question is whether the error term beyond the linear main term is bounded.
\begin{problem}\label{prob:2r+O1}
Does there exist an absolute constant $C>0$ such that
$n_0(r)\le 2r+C$ for all $r\ge 4$?
\end{problem}

For several small values of $r$, namely $r=4,5,\dots,10$, we found explicit $r$-multiplicity antichains on $[n]$ with $n=2r+5$ and
\[
S(\cF)=\{2,3,\dots,n-2\}.
\]
Equivalently, $g(2r+5,r)=(2r+5)-3$ for $4\le r\le 10$.
The corresponding constructions are recorded in~\cite[\texttt{2r\_plus\_5\_r\_4\_to\_10.txt}]{Git}. These examples show that the extremal value $g(n,r)=n-3$ is attained already at $n=2r+5$ for $4\le r\le 10$, which provides some numerical evidence toward Problem~\ref{prob:2r+O1}. Using our current search implementation~\cite[\texttt{antichain\_solver\_one\_based.py}]{Git}, we did not obtain a construction for $r=11$ within 24 hours of runtime.

\appendix

\renewcommand{\theclaim}{\thesection.\arabic{claim}}
\setcounter{claim}{0}

\section{Determining \texorpdfstring{$n_0(2)$ and $n_0(3)$}{n0(2) and n0(3)}}\label{app:small-r}

\subsection{The case \texorpdfstring{$r=2$}{r=2}}\label{app:n02}
We begin with the extremal value in the smallest nontrivial ground set, which will serve as the base case for the threshold statement.
\begin{proposition}\label{prop:g32}
One has $g(3,2)=1$.
\end{proposition}

\begin{proof}
Let $\mathcal F\subseteq 2^{[3]}$ be a $2$-multiplicity antichain. Clearly, levels $0$ and $3$ cannot occur, since $|\binom{[3]}{0}|=|\binom{[3]}{3}|=1<2$. Assume for contradiction that both levels $1$ and $2$ occur, i.e.\ $\{1,2\}\subseteq S(\mathcal F)$.
Then $|\cF_1|\ge2$, so we may choose two distinct singletons $\{x\},\{y\}\in\cF_1$.
Every $2$-subset of $[3]$ contains at least one of $x,y$, hence every $H\in\binom{[3]}{2}$
satisfies $\{x\}\subseteq H$ or $\{y\}\subseteq H$, contradicting the antichain property.
Therefore at most one of the levels $1$ and $2$ can occur, and thus $|S(\cF)|\le1$.
On the other hand, $\cF=\{\{1\},\{2\}\}$ is a $2$-multiplicity antichain with $|S(\cF)|=1$.
Hence $g(3,2)=1$.
\end{proof}
We can now determine the threshold $n_0(2)$ by combining Proposition~\ref{prop:g32} with our general upper bound and an explicit verification for the remaining finite range.

\begin{theorem}\label{thm:n02}
One has $n_0(2)=3$.
\end{theorem}

\begin{proof}
Proposition~\ref{prop:g32} implies the lower bound $n_0(2)\ge 3$.
On the other hand, \eqref{eq:logloglog_upper1} yields the upper bound $n_0(2)\le 21$.
By Definition~\ref{def:threshold}, to prove $n_0(2)=3$ it is enough to verify that
$g(n,2)=n-3$ holds for every integer $n$ with $4\le n\le 21$. Moreover, Lemma~\ref{lem:no-n-2} gives the general upper bound $g(n,2)\le n-3$. Thus it suffices to construct, for each integer $4\le n\le 21$, a $2$-multiplicity antichain on $[n]$
whose occurring set sizes cover exactly $n-3$ distinct levels.
This certifies $g(n,2)\ge n-3$ for all $4\le n\le 21$, and hence $g(n,2)=n-3$ in this range.
The constructions are recorded in our GitHub repository~\cite[\texttt{r2\_n\_4\_to\_21.txt}]{Git}.
\end{proof}

\subsection{The case \texorpdfstring{$r=3$}{r=3}}\label{app:n03}

In this subsection we determine $n_0(3)$. The key input is a small obstruction at $n=8$,
which we obtain by a short structural reduction followed by an exhaustive check.

\begin{proposition}\label{prop:g83}
One has $g(8,3)\le 4$.
\end{proposition}

\begin{proof}
Let $\mathcal F\subseteq 2^{[8]}$ be a $3$-multiplicity antichain. First, levels $0$ and $8$ cannot occur: indeed, $|\binom{[8]}{0}|=|\binom{[8]}{8}|=1$, so neither level can support multiplicity $3$.

\begin{claim}\label{claim:app_a3_step2}
If level $1$ occurs, then $|S(\mathcal F)|\le 4$.
\end{claim}
\begin{proof}[Proof of Claim~\ref{claim:app_a3_step2}]
Assume $1\in S(\mathcal F)$, so $|\mathcal F_1|\ge 3$.
Choose three distinct singletons $\{x\},\{y\},\{z\}\in\mathcal F_1$.
By the antichain property, no other set in $\mathcal F$ may contain any of $x,y,z$. Hence every set in $\mathcal F\setminus \mathcal F_1$ is contained in
\[
[8]\setminus\{x,y,z\},
\]
which has size $5$. In particular, every set in $\mathcal F$ has size at most $5$.
Moreover, at level $5$ there is only one possible set, namely $[8]\setminus\{x,y,z\}$,
so $|\mathcal F_5|\le 1<3$, and thus $5\notin S(\mathcal F)$.
Therefore,
\[
S(\mathcal F)\subseteq \{1,2,3,4\},
\]
and hence $|S(\mathcal F)|\le 4$.
\end{proof}

\begin{claim}\label{claim:app_a3_step3}
If level $7$ occurs, then $|S(\mathcal F)|\le 4$.
\end{claim}
\begin{proof}[Proof of Claim~\ref{claim:app_a3_step3}]
Assume $7\in S(\mathcal F)$, so $|\mathcal F_7|\ge 3$.
Pick three distinct $7$-sets $H_1,H_2,H_3\in\mathcal F_7$. Each can be written as
\[
H_i=[8]\setminus\{a_i\}\qquad (i=1,2,3),
\]
with $a_1,a_2,a_3$ pairwise distinct.
Let $A\in\mathcal F$ with $|A|\le 6$. Since $\mathcal F$ is an antichain, we must have
$A\not\subseteq H_i$ for every $i$, i.e.\ $a_i\in A$ for every $i$. Thus
\[
\{a_1,a_2,a_3\}\subseteq A
\qquad\text{for all }A\in\mathcal F \text{ with }|A|\le 6.
\]
In particular, no set of size $0,1,2$ can belong to $\mathcal F$, and also the only possible $3$-set is $\{a_1,a_2,a_3\}$, so $3\notin S(\mathcal F)$.
Consequently,
\[
S(\mathcal F)\subseteq \{4,5,6,7\},
\]
and hence $|S(\mathcal F)|\le 4$.
\end{proof}

Suppose for contradiction that $|S(\mathcal F)|\ge 5$.
By Claims~\ref{claim:app_a3_step2} and~\ref{claim:app_a3_step3} we must have $1,7\notin S(\mathcal F)$, and also $0,8\notin S(\mathcal F)$. Therefore, the only way to have $|S(\mathcal F)|=5$ is \(S(\mathcal F)=\{2,3,4,5,6\}\). In particular, $|\mathcal F_t|\ge 3$ for each $t\in\{2,3,4,5,6\}$. Choose an arbitrary subfamily $\mathcal F'\subseteq \mathcal F$ by selecting exactly $3$ sets
from each layer $t\in\{2,3,4,5,6\}$. Then $\mathcal F'$ is still an antichain, and moreover
\[
|\mathcal F'_t|=3\qquad \text{for each }t=2,3,4,5,6.
\]
Hence it suffices to rule out the existence of an antichain $\mathcal F'\subseteq 2^{[8]}$
with exactly three sets on each of the five levels $2,3,4,5,6$.

We performed an exhaustive backtracking search over all choices of three sets
from each of the levels $2,3,4,5,6$, pruning only by necessary antichain constraints.
For the top two levels we encode $6$-sets (resp.\ $5$-sets) by their complements,
which are $2$-sets (resp.\ $3$-sets), so that all containment constraints reduce to
explicit intersection/subset tests. The implementation is contained in~\cite[\texttt{check\_g\_8\_3.py}]{Git}; it exhausts all candidates and returns \texttt{None}, meaning that no feasible antichain $\cF'$ exists.

Hence, we conclude that $|S(\mathcal F)|\neq 5$ for every $3$-multiplicity antichain $\mathcal F\subseteq 2^{[8]}$. Therefore $g(8,3)\le 4$, as claimed.
\end{proof}

With Proposition~\ref{prop:g83} in hand, we can now determine the exact threshold $n_0(3)$.

\begin{theorem}\label{thm:n03}
One has $n_0(3)=8$.
\end{theorem}

\begin{proof}
Proposition~\ref{prop:g83} implies the lower bound $n_0(3)\ge 8$.
On the other hand, \eqref{eq:logloglog_upper1} yields the upper bound $n_0(3)\le 24$.
By Definition~\ref{def:threshold}, to prove $n_0(3)=8$ it is enough to verify that
$g(n,3)=n-3$ holds for every integer $n$ with $9\le n\le 24$. Moreover, Lemma~\ref{lem:no-n-2} gives the universal upper bound $g(n,3)\le n-3$ for all $n\ge 4$. Thus it suffices to construct, for each integer $9\le n\le 24$, a $3$-multiplicity antichain on $[n]$
whose occurring set sizes cover exactly $n-3$ distinct levels.
This yields $g(n,3)\ge n-3$ for all $9\le n\le 24$, and hence $g(n,3)=n-3$ in this range.
The explicit constructions are recorded in our GitHub repository~\cite[\texttt{r3\_n\_9\_to\_24.txt}]{Git}.
\end{proof}

\section*{Declaration of generative AI and AI-assisted technologies in the manuscript preparation process}

During the early stages of this project, an AI assistant was used for exploratory brainstorming and preliminary code drafting related to small-$r$ computational checks. All mathematical statements, proofs, and the final presentation of this paper were independently verified and written by the authors.

\end{document}